\documentclass[amsmath,amssymb,12pt]{amsart}

\usepackage{graphicx}

\usepackage{amssymb, amsmath, amsthm, enumerate}

\usepackage{environ}
\newcounter{asyfigcntr}
\makeatletter
\@ifundefined{asy}{%
  \NewEnviron{asy}{%
    \stepcounter{asyfigcntr}%
    \includegraphics{\jobname-\theasyfigcntr}%
  }
}{}

\theoremstyle{plain}
\newtheorem{thm}{Theorem}

\newtheorem{lem}{Lemma}

\newtheorem{cor}{Corollary}

\theoremstyle{definition}

\newcommand{\vll}{\operatorname{area}}

\newcommand{\cnv}{\operatorname{conv}}
\newcommand{\mx}{\text{max}}

\newcommand{\eps}{\varepsilon}

\makeatletter

\begin{document}

\title[Inextensible domains]{Inextensible domains}

\author{Yoav Kallus}
\address{Yoav Kallus, Center for Theoretical Science, Princeton University, Princeton, New Jersey 08544}

\date{\today}

\begin{abstract}
We develop a theory of planar, origin-symmetric, convex domains that are inextensible with
respect to lattice covering, that is, domains such that augmenting them in any way allows
fewer domains to cover the same area. We show that origin-symmetric inextensible domains
are exactly the origin-symmetric convex domains with a circle of outer billiard triangles.
We address a conjecture by Genin and Tabachnikov about convex domains, not necessarily
symmetric, with a circle of outer billiard triangles, and show that it follows immediately
from a result of Sas.
\end{abstract}

\maketitle

In a series of papers from 1946 to 1947, Kurt Mahler developed a theory of planar, origin-symmetric, star-like
domains that are irreducible with respect to lattice packing, that is, domains such that taking away any piece
allows more domains to be packed in the same area \cite{mahler88,mahler95}.
In particular, he was interested in the problem of identifying convex domains such that every convex
domain with lower area may be packed at a greater number per unit area than the domain in question. This
is the subject of Reinhardt's conjecture \cite{Reinhardt}. In particular, Mahler showed that the disk is
not such a domain, even if the domains of lower area to which we compare it are restricted to a small
neighborhood of the disk with respect to Hausdorff distance \cite{mahler95}. 

Inspired by the work of Mahler, we develop a theory of planar, origin-symmetric, convex domains that are 
inextensible with respect to lattice covering, that is, domains such that adding any piece
allows fewer domains to cover the same area. We find that the inextensible domains are simply
those with a circle of outer billiard triangles, a family of domains studied previously by Genin
and Tabachnikov \cite{genin}. The analogue of Reinhardt's domain for covering is simply
the ellipse: any origin-symmetric convex domain can cover the plane with
no larger number per unit area than an ellipse of the same area, 
as can be easily shown using a classical result of Sas \cite{sas,toth-cover}.
In other words, the ellipse covers the plane with the least efficiency.

Genin and Tabachnikov conjecture that out of convex domains, not necessarily symmetric,
with a circle of critical triangles of a fixed area, the ellipse
provides the upper bound for the area of the domain \cite{genin}.
This also follows easily from the result of Sas.

We call $K$ a symmetric convex domain (below just ``domain''), if $K$ is a convex
compact subset of $\mathbb{R}^2$ such that $K=-K$. A lattice $\Lambda=B\mathbb{Z}^2$ is the image
of the integer lattice $\mathbb{Z}^2$ under a nonsingular linear map $B$. The determinant $d(\Lambda)$
of a lattice $\Lambda=B\mathbb{Z}^2$ is given by $|\det(B)|$ and is independent of the basis $B$ used. We will
use the Hausdorff metric as a distance between domains: $\delta(K,K')=\min\{\eps:K\subseteq K'+
\eps D,K'\subseteq K+\eps D\}$, where $D$ is the unit disk. As the distance between lattices we use
the distance between the closest two bases: $\delta(\Lambda,\Lambda')=\min\{||B-B'||:
\Lambda=B\mathbb{Z}^2,\Lambda=B'\mathbb{Z}^2\}$, where $||\cdot||$ is the Hilbert-Schmidt norm.

The lattice $\Lambda$ is called $K$-covering if $K+\Lambda=\mathbb{R}^2$. Of all $K$-covering lattices
there is at least one lattice $\Lambda_c$ such that $d(\Lambda)\le d(\Lambda_c)$
whenever $\Lambda$ is $K$-covering \cite{Gruber}. We call such a lattice a critical lattice for $K$ and its determinant
is called the critical determinant of $K$ and denoted
$\Delta(K)$. Since every point of the unit cell of the lattice is covered at least singly, $\Delta(K)\le\vll(K)$,
with equality if $K$ tiles the plane.
Whenever $K'\supseteq K$, any critical lattice of $K$ is also $K'$-covering, and therefore
$\Delta(K')\ge\Delta(K)$. We say that $K$ is extensible if there is a domain $K'$ containing $K$ but different
from it that has the same critical determinant as $K$. Otherwise, we say $K$ is inextensible.

\begin{figure}
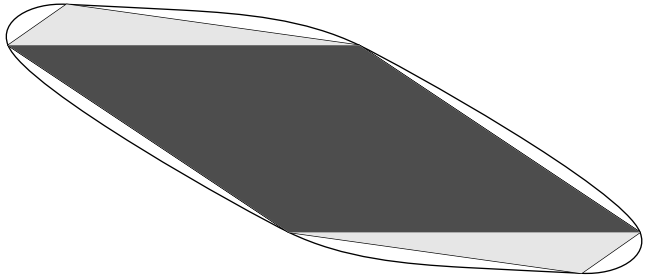
\begin{center}
\begin{asy}
import cseblack;
import olympiad;
usepackage("amssymb");
size(240);

real a=0.3;
real b=.8;

path A=D((-a,-b)--(3-a,-b)--(a,b)--(a-3,b)--cycle);
//path B=D((-a,-b){(-2*a+3,-2*b)}..(-a+1.5,-b-0.6)..(-a+2.,-b-0.7)..{-3+2*a,+2*b}(3-a,-b)--(a,b){(2*a-3,2*b)}..(a-1.5,b+0.6)..(a-2.,b+0.7)..{3-2*a,-2*b}(a-3,b)--cycle);
//path B=D((-a,-b)..(-a+1.5,-b-0.6)..(-a+2.,-b-0.7)..(3-a,-b)..tension atleast 2 ..(a,b)..(a-1.5,b+0.6)..(a-2.,b+0.7)..(a-3,b)..tension atleast 2 ..cycle);
path B=D((-a,-b)..(-a+1.5,-b-0.3)..(-a+2.5,-b-0.35)..(3-a,-b)..tension atleast 2.5 ..(a,b)..(a-1.5,b+0.3)..(a-2.5,b+0.35)..(a-3,b)..tension atleast 2.5 ..cycle);
path C=D((-a,-b)--(-a+2.5,-b-0.35)--(3-a,-b)--(a,b)--(a-2.5,b+0.35)--(a-3,b)--cycle);
//fill(B,lightgray);
fill(C,lightgray);
fill(A,gray(0.3));
\end{asy}
\caption{\label{parafig}
The parallelogram $K$ (dark gray) and a proper superdomain $K'$ (white). Any such superdomain
includes a hexagonal (or degeneratly, parallelogrammatic) superdomain (light gray),
so the parallelogram is inextensible.}
\end{center}\end{figure}

Let us begin with an example of an inextensible domain: a parallelogram (see Figure \ref{parafig}). Since
the parallelogram tiles the plane, the critical
determinant of a parallelogram is equal to its area. Let $K$ be a parallelogram
and let $K'$ be a proper superdomain. By definition there must be a point $\mathbf{x}\in K'\setminus K$.
Let $K''=\cnv(K,\pm\mathbf{x})$. 
Since $K''$ is either a parallelogram of a hexagon, its critical determinant is equal to its area and
therefore $\Delta(K)<\Delta(K'')\le\Delta(K')$.

\begin{figure}
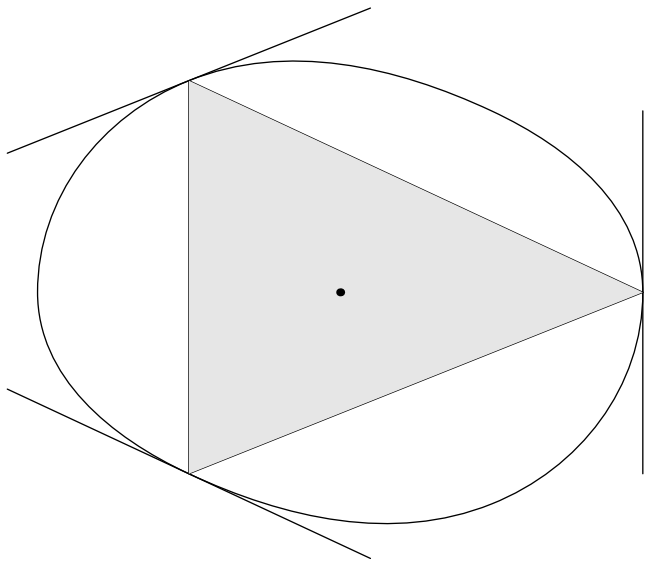
\begin{center}
\begin{asy}
import cseblack;
import olympiad;
usepackage("amssymb");
size(240);

path T=D((1,0)--(-0.5,0.7)--(-0.5,-0.6)--cycle);
path A=D((1,0){0,1}..{-1.5,0.7}(0.5,0.6){-1.5,0.7}..{-1.5,-0.6}(-0.5,0.7){-1.5,-0.6}..{0,-1}(-1,0){0,-1}..{1.5,-0.7}(-0.5,-0.6){1.5,-0.7}..{1.5,0.6}(0.5,-0.7){1.5,0.6}..{0,1}(1,0));
D((1,0.6)--(1,-0.6));
D( ((-0.5,0.7)+0.4*(-1.5,-0.6))--((-0.5,0.7)-0.4*(-1.5,-0.6)) );
D( ((-0.5,-0.6)+0.4*(-1.5,0.7))--((-0.5,-0.6)-0.4*(-1.5,0.7)) );
fill(T,lightgray);
D((0,0));
\end{asy}
\caption{\label{critfig}
A convex, centrally symmetric domain, with one its critical triangles (light gray). The triangle
must contain the origin in its interior (unless the domain is a parallelogram) and there must be
support lines through each of the three vertices that are parallel to the opposite sides of the
triangle.}
\end{center}\end{figure}

Bambah and Rogers have showed that if $T$ is a triangle inscribed in $K$, maximizing the area among all
triangles inscribed in $K$, then the critical determinant of a domain $K$ is equal to twice the area of $T$
\cite{bambah}. We call such a triangle a critical triangle of $K$ (see Figure \ref{critfig}). Consequently, any lattice which
is critical for the (possibly degenerate) hexagon $\cnv(T,-T)$ is critical for $K$.
Any critical triangle $T$ must include the origin, otherwise a reflection of one of its
vertices through the opposite side yields a triangle of equal area but with one vertex in the interior of $K$.
When $K$ is a parallelogram, $T$ may have the origin on its perimeter.
When $K$ is not a parallelogram, the origin must be in the interior of $T$, since otherwise the hexagon
$\cnv(T,-T)$ degenerates to a parallelogram which is simultaneously a proper subdomain of $K$ and of
equal critical determinant, contradicting the inextensibility of parallelograms. Similarly,
the hexagon $\cnv(T,-T)$ must also not degenerate to a parallelogram by having one of the
vertices of $T$ lie on a side of $-T$ and therefore the relative interiors of all sides of $T$ must be interior
to $K$. 

Let $\mathbf{u}_\theta=(\cos\theta,\sin\theta)$ be a point on the unit circle at angle $\theta$ from
the $x$-axis. The support height of a domain $K$ in the direction $\mathbf{u}_\theta$ is given by
$h(\theta)=\max_{\mathbf{x}\in K}\langle\mathbf{x},\mathbf{u}_\theta\rangle$. The
support line $L(\theta)$ is the line $\{\mathbf{x}\in \mathbb{R}^2:\langle\mathbf{x},\mathbf{u}_\theta\rangle = h(\theta)\}$,
and it intersects the boundary of $K$ but not its interior. In order for the critical
triangle to have maximal area among inscribed triangles, there must be a support line
through each of its vertices that is parallel to opposite side (see Figure \ref{critfig}).

\begin{figure}
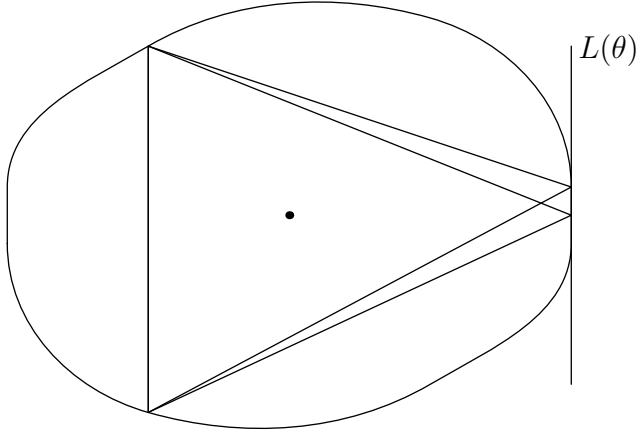
\begin{center}
\begin{asy}
import cseblack;
import olympiad;
usepackage("amssymb");
size(240);

path T1=D((1,-0.0)--(-0.5,-0.7)--(-0.5,0.6)--cycle);
path T2=D((1,0.1)--(-0.5,-0.7)--(-0.5,0.6)--cycle);
//path A=D((-1,-0.1)--(-1,0.1){0,1}..{1.2,0.7}(-0.5,0.6){1.2,0.7}..{1.5,-0.1}(0.5,0.7){1.5,-0.1}..{0,-1}(1,0.1)--(1,-0.1){0,-1}..{-1.2,-0.7}(0.5,-0.6){-1.2,-0.7}..{-1.5,0.1}(-0.5,-0.7){-1.5,0.1}..{0,1}(-1,-0.1));
path A=D((-1,-0.1)--(-1,0.1){0,1}..{1.2,0.7}(-0.5,0.6){1.2,0.7}::{1.5,-0.425}(0.5,0.7){1.5,-0.425}..{0,-1}(1,0.1)--(1,-0.1){0,-1}..{-1.2,-0.7}(0.5,-0.6){-1.2,-0.7}::{-1.5,0.425}(-0.5,-0.7){-1.5,0.425}..{0,1}(-1,-0.1));
MC("{\small L(\theta)}",D((1,0.6)--(1,-0.6)),0.08,NE);
D((0,0));
\end{asy}
\caption{\label{anchorfig}Two triangles anchored at $L(\theta)$: inscribed triangles with one vertex on $L(\theta)$
and one side parallel to $L(\theta)$, maximizing the area among all such triangles. As shown in Lemma \ref{reltri},
the base of all triangles anchored at a given support line is the same. Additionally, if the support line
intersects $K$ at a point, then there is a unique triangle anchored at $L(\theta)$.}
\end{center}\end{figure}

\begin{lem}\label{reltri}Let $L(\theta)$ be a support line of $K$. Consider all triangles $\mathbf{x}\mathbf{y}\mathbf{z}$ inscribed in $K$
in such a way that $\mathbf{x} \in L(\theta)$, and the side $\mathbf{y}\mathbf{z}$ is parallel to $L(\theta)$ (for definiteness,
let $\mathbf{x}\mathbf{y}\mathbf{z}$ be arranged counter-clockwise). There are points $\mathbf{y}_0$ and $\mathbf{z}_0$
such that a triangle of the type described maximizes the area amongst all such triangles if and only if $\mathbf{y}=\mathbf{y}_0$
and $\mathbf{z}=\mathbf{z}_0$.
\end{lem}
\begin{proof}
Let $\mathbf{x}\mathbf{y}\mathbf{z}$ and $\mathbf{x}'\mathbf{y}'\mathbf{z}'$ be two triangles maximizing the area
(see Figure \ref{anchorfig}).
First note that the area does not depend on the point lying on $L(\theta)$. Therefore, we may take $\mathbf{x}=\mathbf{x}'$
without loss of generality.
Let the altitudes of the two triangles taken from $\mathbf{x}$ be $h$ and $h'$, and their area be $A$, so that
their bases have lengths $2A/h$ and $2A/h'$.
By convexity, there is an inscribed triangle with base parallel to $L(\theta)$ of altitude $\tfrac{1}{2}(h+h')$ and
base length at least $A(\tfrac{1}{h}+\tfrac{1}{h'})$. Since its area,
$\tfrac{1}{4}A(2+\tfrac{h'}{h}+\tfrac{h}{h'})$, must not be greater than $A$, 
we have that the altitudes of the two triangles, and their base length, are equal.
If the two bases are not identical, then their convex hull gives the base of an inscribed
triangle of the type described with greater area. Therefore, the bases must be identical.\end{proof}

We denote by $A(\theta)$ the maximal area of a triangle of the type described in Lemma \ref{reltri}, and call
a triangle achieving this area a triangle \textit{anchored} at $L(\theta)$ (or simply at angle $\theta$).
Clearly, the set
of all triangles anchored at support lines of $K$ contains all the critical triangles of $K$ and
$\Delta(K)=2\max_{0 \le \theta < 2 \pi}A(\theta)=2 A_\mx$.
Note that any critical triangle
occurs as a triangle anchored at three angles (parallel to its three sides)
and its symmetric image occurs as a triangle anchored at three more angles.
We label these angles $\theta_i$, $i=1,\ldots,6$, such that $\theta_1<\theta_2<\ldots<\theta_6<\theta_1+2\pi$
and $\theta_{i+3}=\theta_i+\pi$ (the label $i$ is understood to extend cyclically such that $\theta_{i+6}=\theta_i+2\pi$).
We now show that the angles associated with two critical triangle intersperse.

\begin{figure}
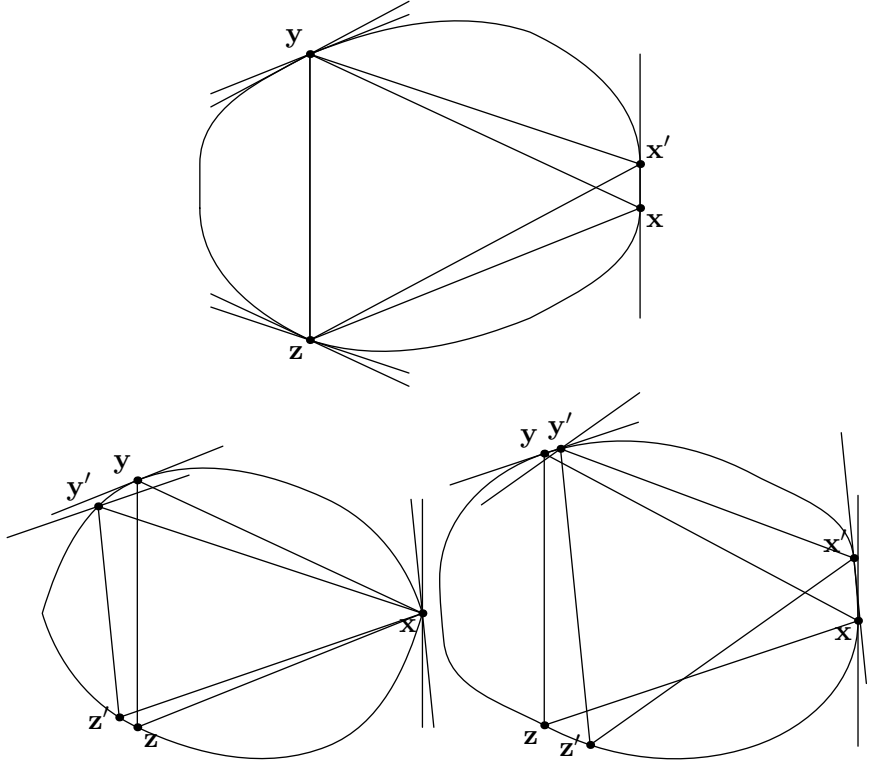
\begin{center}
\begin{asy}[width=0.5\textwidth]
import cseblack;
import olympiad;
usepackage("amssymb");
size(240);

path T1=D((1,-0.1)--(-0.5,-0.7)--(-0.5,0.6)--cycle);
path T2=D((1,0.1)--(-0.5,-0.7)--(-0.5,0.6)--cycle);
pair X=(1,-0.1); pair XP=(1,0.1);
pair Y=(-0.5,0.6); pair Z=(-0.5,-0.7);
path A=D((-1,-0.1)--(-1,0.1){0,1}..{XP-Z}Y{X-Z}..{XP-Y}(-Z){X-Y}..{0,-1}(1,0.1)--(1,-0.1){0,-1}..{Z-XP}(-Y){Z-X}..{Y-XP}(Z){Y-X}..{0,1}(-1,-0.1));
D((1,0.6)--(1,-0.6));
D(Y+0.3(XP-Z)--(Y-0.3(XP-Z)));
D(Y+0.3(X-Z)--(Y-0.3(X-Z)));
D(Z+0.3(XP-Y)--(Z-0.3(XP-Y)));
D(Z+0.3(X-Y)--(Z-0.3(X-Y)));
MP("\mathbf{x}",D((1,-0.1)),SE);
MP("\mathbf{x}'",D((1,0.1)),NE);
MP("\mathbf{z}",D((-0.5,-0.7)),SW);
MP("\mathbf{y}",D((-0.5,0.6)),NW);
\end{asy}
\linebreak
\begin{asy}[width=0.45\textwidth]
import cseblack;
import olympiad;
usepackage("amssymb");
size(240);

real r=0.65;
path T=D((1,0)--(-0.5,0.7)--(-0.5,-0.6)--cycle);
path A=D((1,0){-0.3,1}..{-1.5,0.7}(0.5,0.6){-1.5,0.7}..{-1.5,-0.6}(-0.5,0.7){-1.5,-0.6}..{-0.3,-1}(-1,0){0.3,-1}..{1.5,-0.7}(-0.5,-0.6){1.5,-0.7}..{1.5,0.6}(0.5,-0.7){1.5,0.6}..{0.3,1}(1,0));
pair V1=IP((-r+0.1,-1)--(-r,0),A);
pair V2=IP((-r-0.1,1)--(-r,0),A);
path T2=D((1,0)--V1--V2--cycle);
D((1,0.6)---(1,-0.6));
D((1.06,-0.6)---(0.94,0.6));
D( ((-0.5,0.7)+0.3*(-1.5,-0.6))--((-0.5,0.7)-0.3*(-1.5,-0.6)) );
pair S=V1-(1,0);
D( (V2+0.3*S)--(V2-0.3*S) );
MP("\mathbf{x}",D((1,0)),SW);
MP("\mathbf{y}",D((-0.5,0.7)),NW);
MP("\mathbf{y}'",D(V2),NW);
MP("\mathbf{z}",D((-0.5,-0.6)),SE);
MP("\mathbf{z}'",D(V1),W);
\end{asy}
\begin{asy}[width=0.45\textwidth]
import cseblack;
import olympiad;
usepackage("amssymb");
size(240);

real r=0.35;
path T=D((1,-0.1)--(-0.5,0.7)--(-0.5,-0.6)--cycle);
path A=D((1,-0.1){0,1}..{-0.1,1}(0.98,0.2){-0.1,1}..{-1.5,0.8}(0.5,0.6){-1.5,0.8}..{-1.5,-0.5}(-0.5,0.7){-1.5,-0.5}..{0,-1}(-1,0.1){0,-1}..{0.1,-1}(-0.98,-0.2){0.1,-1}..{1.5,-0.8}(-0.5,-0.6){1.5,-0.8}..{1.5,0.5}(0.5,-0.7){1.5,0.5}..{0,1}(1,-0.1));
pair V1=IP((-r+0.1,-1)--(-r,0),A);
pair V2=IP((-r-0.1,1)--(-r,0),A);
path T2=D((0.98,0.2)--V1--V2--cycle);
D((1,0.5)---(1,-0.7)); //support through x
D((1.04,-0.4)---(0.92,0.8)); //support through xp
D( ((-0.5,0.7)+0.3*(-1.5,-0.5))--((-0.5,0.7)-0.3*(-1.5,-0.5)) ); //support through y
pair S=V1-(0.98,0.2);
D( (V2+0.3*S)--(V2-0.3*S) );
MP("\mathbf{x}'",D((0.98,0.2)),NW);
MP("\mathbf{x}",D((1,-0.1)),SW);
MP("\mathbf{y}",D((-0.5,0.7)),NW);
MP("\mathbf{y}'",D(V2),N);
MP("\mathbf{z}",D((-0.5,-0.6)),SW);
MP("\mathbf{z}'",D(V1),W);
\end{asy}
\caption{\label{interspfig}Three cases of the relative situation
of two critical triangles of a domain, which occur in the proof of Lemma \ref{intersp}.
In the top panel, the triangles are anchored at a common
support line, and necessarily also share two vertices (see Lemma \ref{reltri}). As can be verified by inspection,
the angle triplets at which the two triangles are anchored intersperse. The bottom-left
panel illustrates the impossibility of the case that two critical triangles share a vertex
and the opposite sides do not intersect: the support line through $\mathbf{y}'$ must be parallel
to $\mathbf{x}\mathbf{z}'$, and therefore less steep than the support line through $\mathbf{y}$,
which is parallel to $\mathbf{x}\mathbf{z}$. However, from convexity of $K$, the opposite must also hold.
The bottom-right panel illustrates
the case that two critical triangles share no vertices or support lines, but their
vertices do not intersperse along the boundary. This case is also shown to be
impossible: the support line through $\mathbf{y}'$ must be parallel
to $\mathbf{x}'\mathbf{z}'$, and therefore steeper than the support line through $\mathbf{y}$,
but again, from convexity of $K$, the opposite must also hold.
}
\end{center}\end{figure}

\begin{lem}\label{intersp}Let $\theta_1,\theta_3$, and $\theta_5$ and $\theta_1',\theta_3',$ and $\theta_5'$
be the angles at which two distinct critical
triangles of $K$ arise, then there is some even integer $k$ such that
$\theta_1\le\theta_{k+1}'\le\theta_3\le\theta_{k+3}'\le\theta_5\le\theta_{k+5}'\le\theta_1+2\pi$.
\end{lem}
\begin{proof}
Let $\mathbf{x},\mathbf{y},$ and $\mathbf{z}$ be the vertices incident on the support lines
$L(\theta_1),L(\theta_3),$ and $L(\theta_5)$ respectively, and similarly denote
by $\mathbf{x}',\mathbf{y}',$ and $\mathbf{z}'$ the vertices of the other critical triangle. 
If the two triangles are anchored at a common support line, then by Lemma \ref{reltri}
they also have two vertices in common, and the two remaining vertices lie on a line parallel
to the line connecting the shared vertices. The lemma in that case then follows easily (see top panel
of Figure \ref{interspfig}). Therefore, we may assume below that all six angles are distinct.
We also assume that the domain is not a parallelogram, for which the lemma can also
be easily checked to hold (the details are left to the reader). Therefore,
there is no line segment in the boundary of $K$ containing three of the six vertices.

We will make use of the fact that if $ABCDE$ is a non-degenerate convex pentagon, and
$\vll(ABD)\le\vll(ABE)$ then $\vll(ACD)<\vll(ACE)$. This fact is proved, for example,
in Ref. \cite{dobkin} under the name of the Pentagon Lemma. A more visual argument,
not using this lemma is given in Figure \ref{interspfig}.

First, consider the case that the triangles have a vertex in common,
which without loss of generality may be taken to be $\mathbf{x}=\mathbf{x}'$.
Since the triangles must intersect not only at a vertex (both must contain the origin),
the order of the vertices as we travel counter-clockwise around the boundary
cannot be $\mathbf{x}\mathbf{y}\mathbf{z}\mathbf{y}'\mathbf{z}'\mathbf{x}$ or
$\mathbf{x}\mathbf{y}'\mathbf{z}'\mathbf{y}\mathbf{z}\mathbf{x}$. If the order
is $\mathbf{x}\mathbf{y}\mathbf{y}'\mathbf{z}'\mathbf{z}\mathbf{x}$,
then since $\vll(\mathbf{x}\mathbf{y}\mathbf{z}')\le\vll(\mathbf{x}\mathbf{y}\mathbf{z})$,
the Pentagon Lemma gives that $\vll(\mathbf{x}\mathbf{y}'\mathbf{z}')<\vll(\mathbf{x}\mathbf{y}'\mathbf{z})$:
a contradiction. (see also Figure \ref{interspfig} for a visual
illustration of why this order is impossible.) The same argument, with the triangles transposed also eliminates
the order $\mathbf{x}\mathbf{y}'\mathbf{y}\mathbf{z}\mathbf{z}'\mathbf{x}$.
The order must then be $\mathbf{x}\mathbf{y}\mathbf{y}'\mathbf{z}\mathbf{z}'\mathbf{x}$,
or equivalently $\mathbf{x}\mathbf{y}'\mathbf{y}\mathbf{z}'\mathbf{z}\mathbf{x}$. 
Since we know that the angle of support lines through points on the boundary
cannot decrease as the point is moved counter-clockwise along the boundary, it is 
enough to observe that $\theta_1<\theta_1'$ in the first case, since the corresponding
support lines are parallel to $\mathbf{y}\mathbf{z}$ and $\mathbf{y}'\mathbf{z}'$,
and similarly that $\theta_1>\theta_1'$ in the other case.

Now, consider the case that none of the vertices coincide. We must eliminate all cases
in which one of the three arcs $\mathbf{x}\mathbf{y}$, $\mathbf{y}\mathbf{z}$, or $\mathbf{z}\mathbf{x}$
includes at least two of the vertices $\mathbf{x}'$, $\mathbf{y}'$, or $\mathbf{z}'$.
Since the triangles must intersect, it includes exactly two. Without loss of generally,
we may consider the case where the order is $\mathbf{x}\mathbf{x}'\mathbf{y}'\mathbf{y}\mathbf{z}\mathbf{z}'\mathbf{x}$. 
From the fact that $\vll(\mathbf{x}\mathbf{y}'\mathbf{z})\le\vll(\mathbf{x}\mathbf{y}\mathbf{z})$
and $\vll(\mathbf{x}'\mathbf{y}\mathbf{z}')\le\vll(\mathbf{x}'\mathbf{y}'\mathbf{z}')$,
the Pentagon Lemma gives that
$\vll(\mathbf{x}'\mathbf{y}'\mathbf{z})<\vll(\mathbf{x}'\mathbf{y}\mathbf{z})$ and
$\vll(\mathbf{x}'\mathbf{y}\mathbf{z})<\vll(\mathbf{x}'\mathbf{y}'\mathbf{z})$: a contradiction.
(see also Figure \ref{interspfig} for a visual illustration of why this order is impossible.)
Consequently, the vertices of the two triangles must alternate around the boundary, and the lemma
holds.
\end{proof}


Intuitively, a domain can only be inextensible if it has a one parameter family of
critical lattices. Such a situation is needed and sufficient to make sure that any augmentation of the domain would allow one of the
critical lattices to increase in determinant. We prove this now.

\begin{thm}$K$ is inextensible if and only if $A(\theta)$ is constant.\label{inexthm}\end{thm}

\begin{proof}
First assume that $K$ is extensible. Then there exists a proper superdomain $K'\supset K$ such that
$\max A'(\theta)=\max A(\theta)$. There exists a support line $L(\theta_0)$ of $K$ which intersects
the interior of $K'$. Since the triangle anchored at $L'(\theta_0)$ must be of greater area
than the triangle anchored at $L(\theta_0)$, we have $A(\theta_0)<A'(\theta_0)\le\max A'(\theta)=\max A(\theta)$,
and $A(\theta)$ is not constant.

\begin{figure}
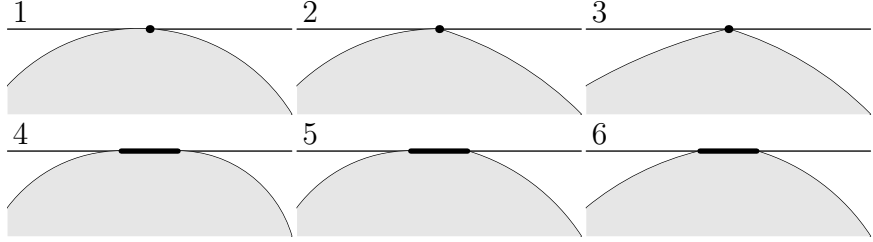
\begin{center}
\begin{asy}[width=0.3\textwidth]
import cseblack;
import olympiad;
usepackage("amssymb");
size(240);

MC("1",D((-1,0)--(1,0)),0.1,NW);
D((-1,-0.4)..(0,0)..(1,-0.6));
fill((-1,-0.6)--(-1,-0.4)..(0,0)..(1,-0.6)--cycle,lightgray);
D((0,0));
\end{asy}
\begin{asy}[width=0.3\textwidth]
import cseblack;
import olympiad;
usepackage("amssymb");
size(240);

MC("2",D((-1,0)--(1,0)),0.1,NW);
D((-1,-0.4)..{1,0}(0,0){1,-0.30}..(1,-0.6));
fill((-1,-0.6)--(-1,-0.4)..{1,0}(0,0){1,-0.30}..(1,-0.6)--cycle,lightgray);
D((0,0));
\end{asy}
\begin{asy}[width=0.3\textwidth]
import cseblack;
import olympiad;
usepackage("amssymb");
size(240);

MC("3",D((-1,0)--(1,0)),0.1,NW);
D((-1,-0.4)..{1,0.24}(0,0){1,-0.30}..(1,-0.6));
fill((-1,-0.6)--(-1,-0.4)..{1,0.24}(0,0){1,-0.30}..(1,-0.6)--cycle,lightgray);
D((0,0));
\end{asy}
\linebreak
\begin{asy}[width=0.3\textwidth]
import cseblack;
import olympiad;
usepackage("amssymb");
size(240);

MC("4",D((-1,0)--(1,0)),0.1,NW);
D((-1,-0.4)..{1,0}(-0.2,0)--(0.2,0){1,0}..(1,-0.6));
fill((-1,-0.6)--(-1,-0.4)..{1,0}(-0.2,0)--(0.2,0){1,0}..(1,-0.6)--cycle,lightgray);
D((-0.2,0)--(0.2,0),linewidth(2));
\end{asy}
\begin{asy}[width=0.3\textwidth]
import cseblack;
import olympiad;
usepackage("amssymb");
size(240);

MC("5",D((-1,0)--(1,0)),0.1,NW);
D((-1,-0.4)..{1,0}(-0.2,0)--(0.2,0){1,-0.3}..(1,-0.6));
fill((-1,-0.6)--(-1,-0.4)..{1,0}(-0.2,0)--(0.2,0){1,-0.3}..(1,-0.6)--cycle,lightgray);
D((-0.2,0)--(0.2,0),linewidth(2));
\end{asy}
\begin{asy}[width=0.3\textwidth]
import cseblack;
import olympiad;
usepackage("amssymb");
size(240);

MC("6",D((-1,0)--(1,0)),0.1,NW);
D((-1,-0.4)..{1,0.24}(-0.2,0)--(0.2,0){1,-0.3}..(1,-0.6));
fill((-1,-0.6)--(-1,-0.4)..{1,0.24}(-0.2,0)--(0.2,0){1,-0.3}..(1,-0.6)--cycle,lightgray);
D((-0.2,0)--(0.2,0),linewidth(2));
\end{asy}
\caption{\label{tangentfig}Six possible types of support lines: (1)
intersecting the boundary at a point and tangent to the boundary on both sides, (2)
intersecting the boundary at a point and tangent on one side, (3) intersecting
the boundary at a point and tangent on neither side, (4) intersecting
the boundary at a segment and tangent on both sides of the segment,
(5) intersecting the boundary at a segment and tangent on one side of the segment,
and (6) intersecting the boundary at a segment and tangent on neither side of the segment.
For the purposes of the proof of Theorem \ref{inexthm}, we call all the above cases
\textit{tangent} except for the case (3).}
\end{center}\end{figure}

For the converse assume that $A(\theta)$ is not constant. Note that $A(\theta)$ must be continuous
in $\theta$, so the set of angles such that $A(\theta)<A_\mx$ is open. Consider those support
lines $L(\theta)$ of $K$ that are not tangent to $K$ on either side of their intersection with the
boundary, that is, support lines such that support lines for sufficiently nearby angles have the same
intersection with the boundary. Call support lines of this type non-tangent
and call all other support lines tangent (see Figure \ref{tangentfig}). 

We assume first that there is an angle $\theta$ such that $L(\theta)$ is
tangent and $A(\theta)<A_\mx$. Since sufficiently nearby angles
also have suboptimal anchored triangles, we must have two support lines $L(\theta')$ and $L(\theta'')$,
whose intersections with the boundary of $K$ are disjoint, and such that for all $\theta\in[\theta',\theta'']$,
we have $A(\theta)<A_\mx$. Therefore, by filling in the area between $K$, $L(\theta')$ and $L(\theta'')$,
we may construct superdomains $K'_\eps\supset K$ with the same support lines
as $K$ for angles outside $[\theta',\theta'']$ (and the symmetric interval $[-\theta',-\theta'']$) and support lines that are 
moved by no more than $\eps$ for the angles in $[\theta',\theta'']$ (see Figure \ref{bumpfig}).
It is clear that $\Delta(K'_\eps)=\Delta(K)$ for small enough $\eps$
and $K$ is extensible. Otherwise, there must be for all $\eps>0$
a critical triangle of $K'_\eps$ of area greater than $A_\mx$ and with a vertex
on the part of the boundary of $K'_\eps$ that is not shared with $K$. The existence
of these triangles would imply the existence of an anchored triangle for $K$
anchored at an angle $\theta\in[\theta',\theta'']$ and with area at least $A_\mx$, and is therefore impossible.

\begin{figure}
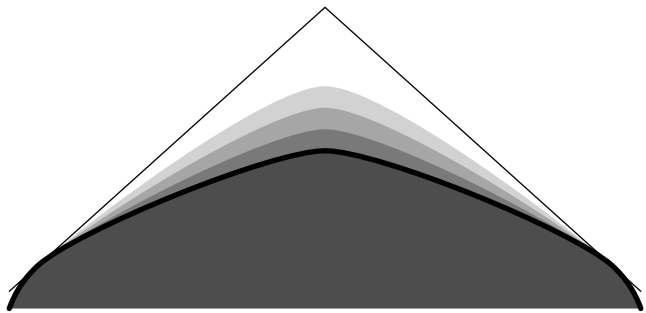
\begin{center}
\begin{asy}
import cseblack;
import olympiad;
usepackage("amssymb");
size(240);

pair X=(-1,0);
pair Y=(1,0);
fill(X{1,0.9}..tension atleast 2 ..(0,0.625)..tension atleast 2 ..{1,-0.9}Y--cycle,gray(0.825));
fill(X{1,0.9}..tension atleast 2 ..(0,0.55)..tension atleast 2 ..{1,-0.9}Y--cycle,gray(0.65));
fill(X{1,0.9}..tension atleast 2 ..(0,0.475)..tension atleast 2 ..{1,-0.9}Y--cycle,gray(0.475));
fill((-1.1,-0.15)..{1,0.9}X{1,0.9}..tension atleast 2 ..(0,0.4)..tension atleast 2 ..{1,-0.9}Y{1,-0.9}..(1.1,-0.15)--cycle,gray(0.3));
D((-1.1,-0.15)..{1,0.9}X{1,0.9}..tension atleast 2 ..(0,0.4)..tension atleast 2 ..{1,-0.9}Y{1,-0.9}..(1.1,-0.15),linewidth(2));
D((-1.1,-0.09)--X--(0,0.9)--Y--(1.1,-0.09));
\end{asy}
\caption{\label{bumpfig}If the intersections
of the support lines $L(\theta')$ and $L(\theta'')$ (thin lines) with the boundary of the
domain $K$ (darkest gray, boundary in thick line) are disjoint,
then we can construct superdomains (shades of gray) whose support lines are identical outside the intervals
$[\theta',\theta'']$ and $[-\theta',-\theta'']$, and arbitrarily close to the original
support lines inside the intervals.}
\end{center}\end{figure}

On the other hand, assume that all angles for which $A(\theta)$ is submaximal
have support lines that are non-tangent. Let us consider an angle $\theta$ for which
$A(\theta)$ is submaximal and denote as $\mathbf{x}$ the point at which $L(\theta)$
intersects $\partial K$. By continuity, $\theta$ must be part of an open interval
$(\theta_1,\theta_1')$ such that $A(\theta_1)=A(\theta_1')=A_\mx$ and $A(\theta')<A_\mx$ for
all $\theta'\in(\theta_1,\theta_1')$. Since all angles $\theta'\in(\theta_1,\theta_1')$
are submaximal, their support lines are non-tangent, and therefore both $L(\theta_1)$
and $L(\theta_1')$ contain $\mathbf{x}\in\partial K$, and $\mathbf{x}$ is a vertex
of two critical triangles anchored at $L(\theta_1)$ and $L(\theta_1')$.
Denote the other two angles anchoring each
of these triangles as $\theta_3$ and $\theta_5$ and $\theta_3'$ and $\theta_5'$.
Consider the intervals $(\theta_3,\theta_3')$ and $(\theta_5,\theta_5')$. By
the interspersing property (Lemma \ref{intersp}), $A(\theta)$ is submaximal
at all angles of these intervals, and therefore all angles of these intervals
must have non-tangent support lines. It follows that all vertices of the two
critical triangles coincide. Since this is a contradiction, the case where all
submaximal angles have non-tangent support lines is impossible.
\end{proof}

Genin and Tabachnikov have observed that convex domains, not necessarily symmetric, such that 
each support line anchors a critical triangle can be characterized by the equivalent
property that each point on its boundary lies on an
outer billiard triangle, that is, a circumscribed triangle such that
the midpoint of each side is on the boundary of the domain \cite{genin}.
For conciseness, we will say that a domain (again, not necessary symmetric)
that possesses these equivalent properties has a circle of critical triangles.
The theorem identifies the family of inextensible symmetric convex domains
with the family of symmetric convex domains with a circle of critical
triangles. Examples of inextensible domains include, in addition to the
parallelogram and ellipse, the regular $6n+4$-gons. Another example of a
family of inextensible domains interpolating between the disk
and the square is given in Figure \ref{inexfig}.

\begin{figure}
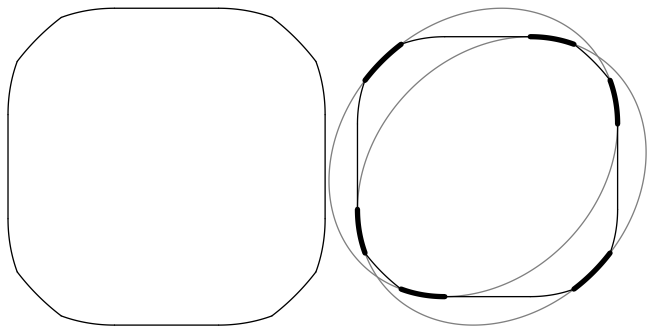
\begin{center}
\begin{asy}
import cseblack;
import olympiad;
usepackage("amssymb");
size(120);

real T=0.0993399*180/3.14159;
pair A=(-0.0996729,0.0996729);
pair B=(-0.0996729,-0.0996729);
real SA=1.105171;
real SB=0.904837;
real X=0.910311;
real Y=0.299019;

D(shift(A)*rotate(45)*scale(SA,SB)*CR((0,0),1,-45+120+T,-45+270-(120+T)));
D(shift(A)*rotate(45)*scale(SA,SB)*CR((0,0),1,-45+T,-45-120+270-(120+T)));
D(shift(A)*rotate(45)*scale(SA,SB)*CR((0,0),1,-45+240+T,-45+120+270-(120+T)));

D(shift(-A)*rotate(45)*scale(SA,SB)*CR((0,0),1,180-45+120+T,180-45+270-(120+T)));
D(shift(-A)*rotate(45)*scale(SA,SB)*CR((0,0),1,180-45+T,180-45-120+270-(120+T)));
D(shift(-A)*rotate(45)*scale(SA,SB)*CR((0,0),1,180-45+240+T,180-45+120+270-(120+T)));

D(shift(B)*rotate(135)*scale(SA,SB)*CR((0,0),1,-45+120+T,-45+270-(120+T)));
D(shift(B)*rotate(135)*scale(SA,SB)*CR((0,0),1,-45+T,-45-120+270-(120+T)));
D(shift(B)*rotate(135)*scale(SA,SB)*CR((0,0),1,-45+240+T,-45+120+270-(120+T)));

D(shift(-B)*rotate(135)*scale(SA,SB)*CR((0,0),1,180-45+120+T,180-45+270-(120+T)));
D(shift(-B)*rotate(135)*scale(SA,SB)*CR((0,0),1,180-45+T,180-45-120+270-(120+T)));
D(shift(-B)*rotate(135)*scale(SA,SB)*CR((0,0),1,180-45+240+T,180-45+120+270-(120+T)));

D((X,Y)--(X,-Y));
D((-X,Y)--(-X,-Y));
D((Y,X)--(-Y,X));
D((Y,-X)--(-Y,-X));
\end{asy}
\begin{asy}
import cseblack;
import olympiad;
usepackage("amssymb");
size(120);

real T=0.0993399*180/3.14159;
pair A=(-0.0996729,0.0996729);
pair B=(-0.0996729,-0.0996729);
real SA=1.105171;
real SB=0.904837;
real X=0.910311;
real Y=0.299019;
D(shift(A)*rotate(45)*scale(SA,SB)*CR((0,0),1,0,360),grey);
D(shift(-A)*rotate(45)*scale(SA,SB)*CR((0,0),1,0,360),grey);

D(shift(A)*rotate(45)*scale(SA,SB)*CR((0,0),1,-45+120+T,-45+270-(120+T)),linewidth(2));
D(shift(A)*rotate(45)*scale(SA,SB)*CR((0,0),1,-45+T,-45-120+270-(120+T)),linewidth(2));
D(shift(A)*rotate(45)*scale(SA,SB)*CR((0,0),1,-45+240+T,-45+120+270-(120+T)),linewidth(2));

D(shift(-A)*rotate(45)*scale(SA,SB)*CR((0,0),1,180-45+120+T,180-45+270-(120+T)),linewidth(2));
D(shift(-A)*rotate(45)*scale(SA,SB)*CR((0,0),1,180-45+T,180-45-120+270-(120+T)),linewidth(2));
D(shift(-A)*rotate(45)*scale(SA,SB)*CR((0,0),1,180-45+240+T,180-45+120+270-(120+T)),linewidth(2));

D(shift(B)*rotate(135)*scale(SA,SB)*CR((0,0),1,-45+120+T,-45+270-(120+T)));
D(shift(B)*rotate(135)*scale(SA,SB)*CR((0,0),1,-45+T,-45-120+270-(120+T)));
D(shift(B)*rotate(135)*scale(SA,SB)*CR((0,0),1,-45+240+T,-45+120+270-(120+T)));

D(shift(-B)*rotate(135)*scale(SA,SB)*CR((0,0),1,180-45+120+T,180-45+270-(120+T)));
D(shift(-B)*rotate(135)*scale(SA,SB)*CR((0,0),1,180-45+T,180-45-120+270-(120+T)));
D(shift(-B)*rotate(135)*scale(SA,SB)*CR((0,0),1,180-45+240+T,180-45+120+270-(120+T)));

D((X,Y)--(X,-Y));
D((-X,Y)--(-X,-Y));
D((Y,X)--(-Y,X));
D((Y,-X)--(-Y,-X));
\end{asy}
\caption{\label{inexfig}
An example of an inextensible domain (left). The boundary of the domain is composed
of 4 line segments and 12 arcs belonging to 4 ellipses. The three arcs
belonging to each ellipse are traced out by the vertices of critical triangles
of that ellipse (right). This construction yields a one-parameter family 
of inextensible domains interpolating between the disk and the square.}
\end{center}\end{figure}

Genin and Tabachnikov conjecture that, among convex domains with
a circle of critical triangles of a given area, the ellipse maximizes
the area of the domain. This in fact follows easily from the following
theorem of Sas:

\begin{thm}(Sas \cite{sas}) Let $K$ be a convex domain and $T_n$ the
inscribed $n$-gon of maximal area in $K$, then
$$\frac{\vll T_n}{\vll K} \ge \frac{n}{2\pi}\sin\frac{2\pi}{n}$$
with equality for, and only for, an ellipse.
\end{thm}

\begin{cor}Let $K$ be a convex domain with a circle of critical triangles
of area $A$, then $\vll K \le 4\pi A/\sqrt{27}$, with equality for,
and only for, an ellipse.
\end{cor}

\bibliographystyle{amsplain}
\bibliography{plane}

\providecommand{\bysame}{\leavevmode\hbox to3em{\hrulefill}\thinspace}
\providecommand{\MR}{\relax\ifhmode\unskip\space\fi MR }
\providecommand{\MRhref}[2]{%
  \href{http://www.ams.org/mathscinet-getitem?mr=#1}{#2}
}
\providecommand{\href}[2]{#2}
\begin{thebibliography}{1}

\bibitem{bambah}
R.~P. Bambah and C.~A. Rogers, \emph{Covering the plane with convex sets},
  Journal of the London Mathematical Society \textbf{s1-27} (1952), no.~3,
  304--314.

\bibitem{dobkin}
D.~P. Dobkin and L.~Snyder, \emph{On a general method for maximizing and
  minimizing among certain geometric problems}, 20th Annual Symposium on
  Foundations of Computer Science, 1979, pp.~9 --17.

\bibitem{genin}
D.~Genin and S.~Tabachnikov, \emph{On configuration spaces of plane polygons,
  sub-{R}iemannian geometry and periodic orbits of outer billiards}, J. Modern
  Dynamics \textbf{1} (2007), 155.

\bibitem{Gruber}
P.~M. Gruber, \emph{Convex and discrete geometry}, Springer, New York, 2007.

\bibitem{mahler88}
K.~Mahler, \emph{On lattice points in n-dimensional star bodies {I}--{IV}},
  Proceedings of the Koninklijke Nederlandse Academie van Wetenschappen
  \textbf{49} (1946), 331--343, 444--454, 524--532, 622--631.

\bibitem{mahler95}
\bysame, \emph{On irreducible convex domains, {O}n the area and the densest
  packing of convex domains, {O}n the minimum determinant and the circumscribed
  hexagons of a convex domain}, Proceedings of the Koninklijke Nederlandse
  Academie van Wetenschappen \textbf{50} (1947), 98--107, 108--118, 326--337.

\bibitem{Reinhardt}
K.~Reinhardt, \emph{{\"U}ber die dichteste gitterf{\"o}rmige {L}agerung
  kongruente {B}ereiche in der {E}bene und eine besondere {A}rt konvexer
  {K}urven}, Abh. Math. Sem., Hamburg, Hansischer Universit{\"a}t, Hamburg
  \textbf{10} (1934), 216.

\bibitem{sas}
E.~Sas, \emph{{\"Uber} eine {Extremumeigenschaft} der {Ellipsen}}, Compositio
  Math. \textbf{6} (1939), 468.

\bibitem{toth-cover}
L.~Fejes T{\'o}th, \emph{Lagerungen in der {Ebene}, auf der {Kugel} und im
  {Raum}}, Springer, Berlin, 1972.

\end{thebibliography}

\end{document}